\documentclass[12pt, reqno]{amsproc}

\usepackage[margin = 1in]{geometry}
\usepackage{amsmath,amsfonts,amssymb,amsthm}
\usepackage[usenames,dvipsnames]{xcolor}
\usepackage[osf]{mathpazo}
\usepackage[euler-digits]{eulervm}
\usepackage{graphicx}
\usepackage{tikz}
\usetikzlibrary{positioning}
\usepackage{tikz-cd}

\usepackage{hyperref}
\hypersetup{
    colorlinks,	%
    citecolor=blue,%
    filecolor=black,%
    linkcolor=red,%
    urlcolor=gray
}

\numberwithin{equation}{section}

	\newtheorem{theorem}[equation]{Theorem}
	\newtheorem{lemma}[equation]{Lemma}
	\newtheorem{corollary}[equation]{Corollary}
	\newtheorem{proposition}[equation]{Proposition}
	\newtheorem*{theorem*}{Theorem}

\theoremstyle{definition}
	\newtheorem{definition}[equation]{Definition}
	\newtheorem{example}[equation]{Example}
        \newtheorem{assumption}[equation]{Assumption}

\theoremstyle{remark}
	
	\newtheorem*{remark*}{Remark}

 { \begin{list}%
         {$\bullet$}%
         {\setlength{\labelwidth}{20pt}%
          \setlength{\leftmargin}{25pt}%
          \setlength{\topsep}{0pt}
          \setlength{\itemsep}{1.5ex}
          \setlength{\parsep}{0pt}}}%
 { \end{list} }

\newcommand{\excise}[1]{}

\newcommand{\R}{\mathbb{R}}

\newcommand{\N}{\mathbb{N}}

\newcommand{\isom}{\cong}

\newcommand{\st}{\ \mid \ }

\newcommand\abs[1]{\lvert#1\rvert}
\newcommand\norm[1]{\left\lVert#1\right\rVert}

\newcommand{\eps}{\varepsilon}

\DeclareMathOperator{\rank}{rank}

\DeclareMathOperator*{\kmax}{kmax}

\begin{document}

\title{The persistence landscape and some of its properties}

\author{Peter Bubenik}
\address{Department of Mathematics, University of Florida}
\email{peter.bubenik@ufl.edu}


\begin{abstract}
  Persistence landscapes map persistence diagrams into a function space, which may often be taken to be a Banach space or even a Hilbert space. In the latter case, it is a feature map and there is an associated kernel. The main advantage of this summary is that it allows one to apply tools from statistics and machine learning. Furthermore, the mapping from persistence diagrams to persistence landscapes is stable and invertible. We introduce a weighted version of the persistence landscape and define a one-parameter family of Poisson-weighted persistence landscape kernels that may be useful for learning. 
We also demonstrate some additional properties of the persistence landscape. First, the persistence landscape may be viewed as a tropical rational function. Second, in many cases it is possible to exactly reconstruct all of the component persistence diagrams from an average persistence landscape. It follows that the persistence landscape kernel is characteristic for certain generic empirical measures. Finally, the persistence landscape distance may be arbitrarily small compared to the interleaving distance.
\end{abstract}

\maketitle
 

\section{Introduction} \label{sec:intro}

A central tool in topological data analysis is persistent homology~\cite{elz:tPaS,zomorodianCarlsson:computingPH} which summarizes geometric and topological information in data using a persistence diagram (or equivalently, a bar code).

For topological data analysis, one wants to subsequently perform statistics and machine learning. There are some approaches to doing so directly with persistence diagrams~\cite{bubenikKim:parametric,bckl:nonparametric,Blumberg:2014,mmh:probability,Robinson:2017}. However, using the standard metrics for persistence diagrams (bottleneck distance and Wasserstein distance) it is difficult to even perform such a basic statistical operation as averaging~\cite{tmmh:frechet-means,munch:probabilistic-f}.

The modern approach to alleviating these difficulties and to permit the easy application of statistical and machine learning methods is to map persistence diagrams to a Hilbert space. One way to do so is the persistence landscape~\cite{bubenik:landscapes}. It has the advantages of being invertible, so it does not lose any information, having stability properties, and being parameter-free and nonlinear (see Section~\ref{sec:pl}). 

The persistence landscape may be efficiently computed either exactly or using a discrete approximation~\cite{bubenikDlotko}. Since it loses no information (or little information in the case of the discrete approximation) it can be a large representation of the persistence diagram. Nevertheless, subsequent statistical and machine learning computations are fast, which has allowed a wide variety of applications. These include the study of:
electroencephalographic signals~\cite{Wang:2015a,Wang:2018},
protein binding~\cite{giseon:maltose},
microstructure analysis~\cite{Dlotko:2016},
phase transitions~\cite{Donato:2016},
swarm behavior~\cite{Corcoran:2016},
nanoporous materials~\cite{Lee:2017,Lee:2018}, 
fMRI data~\cite{Stolz:2017,Stolz:2018},
coupled oscillators~\cite{Stolz:2017},
brain geometry~\cite{Garg:2017,Garg:2017a},
detecting financial crashes~\cite{Gidea:2018},
shape analysis~\cite{Patrangenaru:2018},
histology images~\cite{Chittajallu:2018},
music audio signals~\cite{Liu:2016}, and
the microbiome~\cite{Petrov:2017}.

In this paper we introduce a weighted version of the persistence landscape (Section~\ref{sec:weighted}). In some applications it has been observed that it is not the longest bars that are the most relevant, but those of intermediate length~\cite{bendich:brain-artery,Patrangenaru:2018}. 
The addition of a weighting allows one to tune the persistence landscape to emphasize the feature scales of greatest interest. Since arbitrary weights allow perhaps too much flexibility, we introduce the Poisson-weighted persistence landscape kernel which has one degree of freedom.

Next we show that persistence landscapes are highly compatible with Kalisnik's tropical rational function approach to summarizing persistent homology~\cite{Kalisnik:2018}. In fact, we show that persistence landscapes are tropical rational functions (Section~\ref{sec:tropical}).

In the most technical part of the paper (Section~\ref{sec:average}), we prove that for certain finite sets of persistence diagrams, it is possible to recover these persistence diagrams exactly from their average persistence landscape (Theorem~\ref{thm:reconstruct}). Furthermore, we show that this situation is in some sense generic (Theorem~\ref{thm:generic}). This implies that the persistence landscape kernel is characteristic for certain generic empirical measures (Theorem~\ref{thm:characteristic}).

It is known that the $L^{\infty}$ distance between the two persistence landscapes associated to two persistence diagrams is upper bounded by the corresponding bottleneck distance~\cite[Theorem 13]{bubenik:landscapes}.
In the other direction, we show that this $L^{\infty}$ distance is not lower bounded by some fixed positive scalar multiple of the corresponding bottleneck distance (Section~\ref{sec:metric}).

\subsubsection*{Related work}

There are also many other ways to map persistence diagrams to a vector space or Hilbert space. These include
the Euler characteristic curve~\cite{turner:2013},
the persistence scale-space map~\cite{Reininghaus:2015},
complex vectors~\cite{DiFabio:2015},
pairwise distances~\cite{Carriere:2015},
silhouettes~\cite{cflrw:silhouettes},
the longest bars~\cite{bendich:brain-artery},
the rank function~\cite{Robins:2016},
the affine coordinate ring~\cite{adcock:2016},
the persistence weighted Gaussian kernel~\cite{Kusano:2016},
topological pooling~\cite{Bonis:2016},
the Hilbert sphere~\cite{Anirudh:2016},
persistence images~\cite{Adams:2017},
replicating statistical topology~\cite{Adler:2017a},
tropical rational functions~\cite{Kalisnik:2018},
death vectors~\cite{Patrangenaru:2018},
persistence intensity functions~\cite{Chen:2015},
kernel density estimates~\cite{Phillips:2015,Mike:2018},
the sliced Wasserstein kernel~\cite{Carriere:2017a},
the smooth Euler characteristic transform~\cite{Crawford:2016},
the accumulated persistence function~\cite{Biscio:2016},
the persistence Fisher kernel~\cite{Le:2018}, 
persistence paths~\cite{Chevyrev:2018}, and
persistence contours~\cite{Riihimaki:2018}.
Perhaps since the persistence diagram is such a rich invariant, it seems that any reasonable way of encoding it in a vector works fairly well.

\subsubsection*{Outline of the paper}

In Section~\ref{sec:background} we recall necessary background information. The next three sections contain our main results.
In Section~\ref{sec:weighted} we define the weighted persistence landscape and the Poisson-weighted persistence landscape kernel.
In Section~\ref{sec:tropical} we show that the persistence landscape may be viewed as a tropical rational function.
In Section~\ref{sec:average} we show that in a certain generic situation we are able to reconstruct a family of persistence diagrams from their average persistence landscape. 
From this it follows that the persistence landscape kernel is characteristic for certain generic empirical measures.
Finally in Section~\ref{sec:metric} we show that the $L^{\infty}$ landscape distance is not lower bounded by a fixed positive scalar multiple of the bottleneck distance.

\section{Background} \label{sec:background}

\subsection{Persistence modules, persistence diagrams, and bar codes} \label{sec:persistence}

A \emph{persistence module}~\cite{bubenikScott:1} $M$ consists of a vector space $M(a)$ for each real number $a$, and for each $a \leq b$ a linear map $M(a\leq b): M(a) \to M(b)$ such that for $a \leq b \leq c$, $M(b \leq c) \circ M(a \leq b) = M(a \leq c)$.
Persistence modules arise in topological data analysis from homology (with coefficients in some field) of a filtered simplicial complex (or a filtered topological space).

In many cases, a persistence module can be completely represented by a collection of intervals called a \emph{bar code}~\cite{Collins:2004}. Another representation of the bar code is the \emph{persistence diagram}~\cite{cseh:stability} consisting of pairs $\{(a_j,b_j)\}_{j \in J}$ which are the end points of the intervals in the bar code.

In computational settings there are always only finitely many points in the persistence diagram and it is usually best to truncate intervals in the bar code that persist until the maximum filtration value at that value.
Thus we make the following assumption.

\begin{assumption} \label{ass:finite-intervals}
Throughout this paper, we will assume that persistence diagrams consist of finitely many points $(b,d)$ with $-\infty< b < d < \infty$.
\end{assumption}
 
One way of measuring distance between persistence modules is the \emph{interleaving distance}~\cite{ccsggo:interleaving}. Similarly, one can measure distance between persistence diagrams is the \emph{bottleneck distance}~\cite{cseh:stability}. The two distances are related by the \emph{isometry theorem}~\cite{ccsggo:interleaving,Lesnick:2015,bubenikScott:1}. These distances induce a topology on the space of persistence modules and the space of persistence diagrams~\cite{bubenikVergili}.  

Sometimes we will consider \emph{sequences of persistence diagrams} $D_1,\ldots,D_n$ for fixed $n$. When we do so, we will consider this sequence to be a point $(D_1,\ldots,D_n)$ in the product space of $n$ persistence diagrams with the product metric.
That is, 
\begin{equation} \label{eq:product-metric}
d\left((D_1,\ldots,D_n),(D'_1,\ldots,D'_n)\right) = \max \left\{ d_B(D_1,D'_1), \ldots, d_B(D_n,D'_n) \right\}. 
\end{equation}
This metric induces the product topology.

\subsection{Persistence landscapes and average persistence landscapes} \label{sec:pl}

We give three equivalent definitions of the persistence landscape~\cite{bubenik:landscapes}.

Given a persistence module, $M$, we may define the \emph{persistence landscape} as the function $\lambda: \N \times \R \to \R$ given by
\begin{equation*}
  \lambda(k,t) = \sup ( h \geq 0 \st \rank M(t-h \leq t+h) \geq k).
\end{equation*}
More concretely, for a bar code, $B = \{I_j\}$, we can define the persistence landscape by
\begin{equation*}
  \lambda(k,t) = \sup ( h \geq 0 \st [t-h,t+h] \subset I_j \text{ for at least $k$ distinct }j).
\end{equation*}
For a persistence diagram $D = \{(a_i,b_i)\}_{i \in I}$, we can define the persistence landscape as follows.
First, for $a < b$, define
\begin{equation*}
  f_{(a,b)}(t) = \max(0, \min(a+t,b-t)).
\end{equation*}
Then
\begin{equation*}
  \lambda(k,t) = \kmax \, \{f_{(a_i,b_i)}(t)\}_{i \in I},
\end{equation*}
where $\kmax$ denotes the $k$th largest element.

The persistence landscape may also be considered to be a sequence of functions $\lambda_1,\lambda_2,\ldots:\R \to \R$, where $\lambda_k$ is called the $k$th \emph{persistence landscape function}.
The function $\lambda_k$ is piecewise linear with slope either $0$, $1$, or $-1$. The \emph{critical points} of $\lambda_k$ are those values of $t$ at which the slope changes.
The set of \emph{critical points} of the persistence landscape $\lambda$ is the union of the sets of critical points of the functions $\lambda_k$.
A persistence landscape may be computed by finding its critical points and also encoded by the sequences of critical points of the persistence landscape functions~\cite{bubenikDlotko}. 

The \emph{average persistence landscape}~\cite{bubenik:landscapes,cflrw:silhouettes} of the persistence landscapes $\lambda^{(1)},\ldots,\lambda^{(N)}$ is given by 
\begin{equation*}
  \bar{\lambda}(k,t) = \frac{1}{N} \sum_{j=1}^N \lambda^{(j)}(k,t).
\end{equation*}
We can also consider $\bar{\lambda}$ to be given by a sequence of functions $\bar{\lambda}_k = \frac{1}{N} \sum_{j=1}^N \lambda_k^{(j)}(t)$.

\subsection{Feature maps and kernels} \label{sec:kernel}

Let $\mathcal{S}$ be a set. A function $F: \mathcal{S} \to \mathcal{H}$ where $\mathcal{H}$ is a Hilbert space is called a \emph{feature map}.
A \emph{kernel} on $\mathcal{S}$ is a symmetric map
$K:\mathcal{S} \times \mathcal{S} \to \R$
such that for every $n$ and all $x_1,\ldots,x_n \in \mathcal{S}$ and  $a_1,\ldots,a_n \in \R$, $\sum_{i,j=1}^n a_ia_jK(x_i,x_j) \geq 0$.
A \emph{reproducing kernel Hilbert space (RKHS)} on a set $\mathcal{S}$ is a Hilbert space of real-valued functions on $\mathcal{S}$ such that the pointwise evaluation functional is continuous. 

Given a feature map there is an associated \emph{kernel} given by
\begin{equation*}
  K(x,y) = \langle(F(x),F(y) \rangle_{\mathcal{H}}.
\end{equation*}
Given a kernel, $K$, there is an associated \emph{reproducing kernel Hilbert space (RKHS)}, $\mathcal{H}_K$, which is the completion of the span of the functions $K_x: \mathcal{S} \to \R$ given by $K_x(y) = K(x,y)$, for all $x \in \mathcal{S}$, with respect to the inner product given by $\langle K_x, K_y \rangle = K(x,y)$.

Now assume that we have a $\sigma$-algebra $\mathcal{A}$ on $\mathcal{S}$. 
One can map measures on $(\mathcal{S},\mathcal{A})$ to $\mathcal{H}_K$ via the map $\Phi_K: \mu \mapsto \int_{\mathcal{S}} K(\cdot,x) \, d\mu(x)$ (when this is well defined). This map is called the \emph{kernel mean embedding}. Let $\mathcal{M}$ be a set of measures on $\mathcal{S}$. 
The kernel $K$ is said to be \emph{characteristic} over $\mathcal{M}$ if the kernel mean embedding is injective on $\mathcal{M}$.

\subsection{Properties of the persistence landscape} \label{sec:properties}

We recall some established properties of the persistence landscape.

\subsubsection{Invertibility}

The following is given informally in~\cite[Section 2.3]{bubenik:landscapes}. It is proved more formally and precisely in~\cite{bbe:graded-pd} where it is shown that the critical points of the persistence landscapes are obtained from a graded version of the rank function via M\"obius inversion.

\begin{theorem}
  The mapping from persistence diagrams to persistence landscapes is invertible.
\end{theorem}

\subsubsection{Stability} \label{sec:stability}

The persistence landscape is stable in the following sense.
 
\begin{theorem}[{\cite[Theorem 13]{bubenik:landscapes}}] \label{thm:stability-bottleneck}
  Let $D$ and $D'$ be two persistence diagrams and let $\lambda$ and $\lambda'$ be their persistence landscapes. Then for all $k$ and all $t$,
  \begin{equation*}
    \abs{\lambda_k(t) - \lambda'_k(t)} \leq d_B(D,D'),
  \end{equation*}
where $d_B$ denotes the bottleneck distance.
\end{theorem}

More generally, we have the following.
\begin{theorem}[{\cite[Theorem 17]{bubenik:landscapes}}] 
  Let $M$ and $M'$ be two persistence modules and let $\lambda$ and $\lambda'$ be their persistence landscapes. Then for all $k$ and all $t$,
  \begin{equation*}
    \abs{\lambda_k(t) - \lambda'_k(t)} \leq d_i(M,M'),
  \end{equation*}
where $d_i$ denotes the interleaving distance.
\end{theorem}

As a special case of Theorem~\ref{thm:stability-bottleneck}, we have the following.
\begin{corollary}
Given the two persistence diagrams $D = \{(a_1,b_1),\ldots,(a_n,b_n)\}$ and $D' = \{(a'_1,b'_1),\ldots,(a'_n,b'_n)\}$,
let $\lambda$ and $\lambda'$ be the associated persistence landscapes.
Then
  \begin{equation*}
    \norm{\lambda - \lambda'}_{\infty}
    \leq \norm{(a_1,b_1,\ldots,a_n,b_n)-(a'_1,b'_1,\ldots,a'_n,b'_n)}_{\infty}.
  \end{equation*}
\end{corollary}

In~\cite{Chazal:2014c} it is shown that the average persistence landscape is stable.

\subsubsection{The persistence landscape kernel}
\label{sec:pers-lansc-kern}
 
Since the persistence landscape is a feature map from the set of persistence diagrams to $L^2(\N \times \R)$ there is an associated kernel we call the \emph{persistence landscape kernel}~\cite{Reininghaus:2015}, given by
\begin{equation} \label{eq:pl-kernel}
  K(D^{(1)},D^{(2)}) = \langle \lambda^{(1)}, \lambda^{(2)} \rangle = \sum_{k} \int \lambda_k^{(1)} \lambda_k^{(2)} = \sum_{k=1}^{\infty} \int_{-\infty}^{\infty} \lambda_k^{(1)}(t) \lambda_k^{(2)}(t) \, dt.
\end{equation}

\subsubsection{The persistence landscapes and parameters}

One advantage of the persistence landscape is that its definition involves no parameters. So there is no need for tuning and no risk of overfitting.

\subsubsection{Nonlinearity of persistence landscapes} \label{sec:nonlinear}

Another important advantage of the persistence landscape for statistics and machine learning is its nonlinearity. 
Call a summary $S$ of persistence diagrams in a vector space \emph{linear} if for two persistence diagrams $D_1$ and $D_2$, $S(D_1 \amalg D_2) = S(D_1) + S(D_2)$.
The persistence landscape is highly non-linear.

\subsubsection{Computability of the persistence landscape}

There are fast algorithms and software for computing the persistence landscape~\cite{bubenikDlotko}. In practice, computing the persistence diagram seems to always be slower than computing the associated persistence landscape.
The methods are also available in an R package~\cite{Bouza:2018}.

\subsubsection{Convergence results for the persistence landscape}

From the point of view of statistics, we assume that data has been obtained by sampling from a random variable. Applying our persistent homology constructions, we obtain a random persistence landscape. 

This is a Banach space valued random variable. Assume that its norm has finite expectation and variance. If we take an (infinite) sequence of samples from this random variable then the average landscapes converge (almost surely) to the expected value of the random variable~\cite[Theorem 9]{bubenik:landscapes}. This is known as a (strong) law of large numbers.

Now if we consider the difference between the average landscapes and the expectation (suitably normalized), it converges pointwise to a Gaussian random variable~\cite[Theorem 10]{bubenik:landscapes}. This result was extended in~\cite{cflrw:silhouettes} to prove uniform convergence. These are central limit theorems.

\subsubsection{Confidence bands for the persistence landscape}

The bootstrap can be used to compute confidence bands~\cite{Chazal:2014a} and adaptive confidence bands~\cite{cflrw:silhouettes} for the persistence landscape. There is an R package that has implemented these computations~\cite{fasy:tda}.

\subsubsection{Subsampling and the average persistence landscape}

A useful and powerful method in large data settings is to subsample many times and compute the average persistence landscape~\cite{Chazal:2014c,Patrangenaru:2018}.
In~\cite{Chazal:2014c} it is shown that this average persistence landscape is stable and that it converges.

\subsection{Tropical rational functions} \label{sec:tropical-background}

The \emph{max-plus algebra} is the semiring over $\R \cup {-\infty}$ with the binary operations given by
\begin{align*}
  x \oplus y = \max(x,y),\\
  x \odot y = x + y.
\end{align*}
If $x_1,\ldots,x_n$ are variables representing elements in the max-plus algebra, then a product of these variables (with repetition allowed) is a \emph{max-plus monomial}.
\begin{equation*}
  x_1^{a_1} x_2^{a_2} \cdots x_n^{a_n}
  =  x_1^{a_1} \odot x_2^{a_2} \odot \cdots \odot x_n^{a_n}
\end{equation*}
A \emph{max-plus polynomial} is a finite linear combination of max-plus monomials.
\begin{equation*}
  p(x_1,\ldots,x_n) = a_1 \odot x_1^{a_1^1} x_2^{a_2^1} \cdots x_n^{a_n^1}
  \oplus a_2 \odot x_2^{a_1^2} x_2^{a_2^2} \cdots x_n^{a_n^2} \oplus \cdots
  \oplus a_m \odot x_1^{a_1^m} x_2^{a_2^m} \cdots x_n^{a_n^m}
\end{equation*}
We also call this a \emph{tropical polynomial}. 
A \emph{tropical rational function}~\cite{Kalisnik:2018} is a quotient $p \odot q^{-1}$ where $p$ and $q$ are tropical polynomials.
Note that if $r$ and $s$ are tropical rational functions, then so is $r \odot s^{-1}$.

\section{Weighted persistence landscapes} \label{sec:weighted}

In this section we introduce a class of norms and kernels for persistence landscapes. As a special case we define a one-parameter family of norms and kernels for persistence landscapes which may be useful for learning algorithms.

Recall that for real-valued functions on $\N \times \R$ we have a $p$-norm for $1 \leq p \leq \infty$. For persistence landscapes, we have for $1 \leq p < \infty$,
\begin{equation*}
  \norm{\lambda}_p = \sum_{k=1}^{\infty} \left[ \int_{-\infty}^{\infty} \lambda_k(t)^p \, dt \right]^{\frac{1}{p}},
\end{equation*}
and for $p=\infty$,
\begin{equation*}
  \norm{\lambda}_{\infty} = \sup_{k,t} \lambda_k(t).
\end{equation*}
We also have the persistence landscape kernel, introduced in \eqref{eq:pl-kernel}, given by the inner product on $\N \times \R$,
\begin{equation*}
  K(D^{(1)},D^{(2)}) = \langle \lambda^{(1)}, \lambda^{(2)} \rangle = \sum_{k} \int \lambda_k^{(1)} \lambda_k^{(2)} = \sum_{k=1}^{\infty} \int_{-\infty}^{\infty} \lambda_k^{(1)}(t) \lambda_k^{(2)}(t) \, dt.
\end{equation*}

We observe that one may use weighted versions of these norms and inner products.
That is,
given any nonnegative function $w:\N \times \R \to \R$, we have
\begin{equation*}
  \norm{\lambda}_{p,w} = \norm{w \lambda}_p,
\end{equation*}
and
\begin{equation*}
  K_w(D^{(1)},D^{(2)}) = \langle w^{\frac{1}{2}} \lambda^{(1)}, w^{\frac{1}{2}} \lambda^{(2)} \rangle.
\end{equation*}

For example, consider the following one-parameter family of kernels,
\begin{equation*}
  K_{\nu}(D^{(1)},D^{(2)}) = \sum_{k=1}^{\infty} P_{\nu}(k-1) \int_{-\infty}^{\infty} \lambda_k^{(1)}(t) \lambda_k^{(2)}(t)\, dt,
\end{equation*}
where $P_{\nu}(k) = \frac{\nu^k e^{-\nu}}{k!}$ is the Poisson distribution with parameter $\nu>0$.
Call this the \emph{Poisson-weighted persistence landscape kernel}.
This additional parameter may be useful for training classifiers using persistence landscapes.
It has an associated one-parameter family of norms given by,
\begin{equation*}
  \norm{\lambda}_{\nu} = \sum_{k=1}^{\infty} P_{\nu}(k-1) \norm{\lambda_k}_2.
\end{equation*}

Note that the distribution $P_{\nu}(k-1)$ is unimodal with maximum at $k = \lceil \nu \rceil$ and $k = \lfloor \nu \rfloor + 1$.
So by varying $\nu$ one increases the weighting of a particular range of persistence landscape functions.

We may consider the kernel $K_{\nu}$ to be associated to the feature map $D \mapsto \lambda(D)$ which maps to the Hilbert space with inner product $\langle f, g \rangle_{\nu} = \sum_k P_{\nu}(k-1) \int f_k g_k$ or the feature map $D \mapsto \sum_k \left(P_{\nu}(k-1)\right)^{\frac{1}{2}} \lambda_k(D)$ which maps to the usual Hilbert space $L^2(\N \times \R)$.

The Poisson distribution was chosen because it provides a one-parameter family of unimodal weights whose modes include all natural numbers. Other one-parameter, few-parameter, or more general weighting schemes may be useful depending on the application and the machine learning methods that are used.
For related work, consider contour learning~\cite{Riihimaki:2018}.

\section{Persistence landscapes as tropical rational functions} \label{sec:tropical}

In this section we will show that the persistence landscape is a tropical rational function.

Let $D = \{(a_i,b_i)\}_{i=1}^n$ be a persistence diagram. Recall that $-\infty < a_i < b_i < \infty$ (Assumption~\ref{ass:finite-intervals}). Recall (Section~\ref{sec:pl}) that the $k$th persistence landscape function is given by $\lambda_k(t) = \kmax_{1 \leq i \leq n} f_{(a_i,b_i)}(t)$, where $f_{(a,b)}(t) = \max(0, \min(a+t,b-t))$.

First rewrite $f_{(a,b)}$ as a tropical rational expression in one variable, $t$, as follows.

\begin{align*}
  f_{(a,b)}(t) &= \max(0, \min(a+t,b-t))\\
               &= \max (0, -\max(-(a+t),t-b))\\
               &= \max(0, -\max ( (a \odot t)^{-1}, t \odot b^{-1}) )\\
               &= \max ( 0, [(a \odot t)^{-1} \oplus (t \odot b^{-1})]^{-1} )\\
  &= 0 \oplus \left[ (a \odot t)^{-1} \oplus (t \odot b^{-1}) \right]^{-1}
\end{align*}
We may simplify the right hand term 
by using the usual rules for adding fractions.%
\footnote{That is, $\left( \frac{1}{at} + \frac{t}{b} \right)^{-1} = \frac{bat}{b+at^2}$.}
So
\begin{equation*}
  f_{(a,b)}(t) = 0 \oplus (a+b) \odot t \odot (b \oplus a \odot t^2)^{-1}.
\end{equation*}

Next consider max-plus polynomials in $n$ variables, $x_1,\ldots,x_n$.
The elementary symmetric max-plus polynomials, $\sigma_1,\ldots,\sigma_n$, are given by
\begin{equation*}
  \sigma_k(x_1,\ldots,x_n) = \oplus_{\pi \in S_n} x_{\pi(1)} \odot \cdots \odot x_{\pi(k)},
\end{equation*}
where the sum is taken over elements $\pi$ of the symmetric group $S_n$.
So $\sigma_k$ is the sum of the $k$th largest elements of $x_1,\ldots,x_n$.
Therefore,
\begin{equation*}
  \kmax_{1\leq i \leq n} x_i = \sigma_k(x_1,\ldots,x_n)  - \sigma_{k-1}(x_1,\ldots,x_n).
\end{equation*}
Thus,
\begin{equation*}
  \lambda_k(t) = \sigma_k(f_i(t)) \odot \sigma_{k-1}(f_i(t))^{-1},
\end{equation*}
where we have written $\sigma_k(x_i)$ for $\sigma_k(x_1,\ldots,x_n)$ and $f_i(t)$ for $f_{(a_i,b_i)}(t)$.
Hence, for a fixed persistence diagram $D$, we have $\lambda_k$ as a tropical rational function in one variable $t$.

However, we really want to consider $t$ as fixed and the persistence diagram as the variable. Let us change to this perspective.
To start,
consider
\begin{equation*}
  f_t(a,b) = 0 \oplus t \odot a \odot b \odot (b \oplus 2t \odot a)^{-1},
\end{equation*}
a tropical rational function in the variables $a$ and $b$.
Next,
\begin{equation*}
  \sigma_k(f_t(a_1,b_1),\ldots,f_t(a_n,b_n)) = \oplus_{\pi \in S_n} f_t(a_{\pi(1)},b_{\pi(1)}) \odot \cdots \odot f_t(a_{\pi(k)},b_{\pi(k)})
\end{equation*}
is a 2-symmetric max-plus tropical rational function in the variables $a_1,b_1,\ldots,a_n,b_n$.
Finally,
\begin{equation*}
  \lambda_{k,t}(a_1,b_1,\ldots,a_n,b_n) = \sigma_k(f_t(a_1,b_1),\ldots,f_t(a_n,b_n)) \odot
\sigma_{k-1}(f_t(a_1,b_1),\ldots,f_t(a_n,b_n))^{-1}
\end{equation*}
is also a 2-symmetric tropical rational function in the variables $a_1,b_1,\ldots,a_n,b_n$.

By the stability theorem for persistence landscapes (Section~\ref{sec:stability}), these tropical rational functions are 1-Lipschitz function from $\R^{2n}$ with the sup-norm to $\R$.

Since the mapping from persistence diagrams to persistence landscapes is invertible~\cite{bubenik:landscapes}, the persistence landscape gives us a collections of tropical rational functions $\lambda_{k,t}$ from which we can reconstruct the persistence diagrams.

In practice, we do not need to use all of the $\lambda_{k,t}$. If the values of $a_i$ and $b_i$ are only known up to some $\eps$ or if they lie on a grid of step size $2\eps$, then it suffices to use $k = 1,\ldots, K$ and $t = a, a+\eps, a+2\eps, a+2m\eps$, where $K$ is the maximal dimension of the persistence module (i.e. the maximum number of overlapping intervals in the bar code), and the interval $[a,a+2m\eps]$ contains all of the $a_i$ and $b_i$.

\section{Reconstruction of diagrams from an average persistence landscape} \label{sec:average}

In this section we will show that for certain generic finite sets of persistence diagrams, it is possible to reconstruct these sets of persistence diagrams exactly from their average persistence landscapes.
This implies that the persistence landscape kernel is characteristic for certain generic empirical measures.

Let $D_1,\ldots,D_n$ be a sequence of persistence diagrams (Section~\ref{sec:persistence}). Recall that we assume that our persistence diagrams consist of finitely many points $(b,d)$ where $\infty < b < d < \infty$ (Assumption~\ref{ass:finite-intervals}).
Let $\lambda(D_1),\ldots,\lambda(D_n)$ denote their corresponding persistence landscapes (Section~\ref{sec:pl}) and let $\bar{\lambda} := \frac{1}{n} \sum_{k=1}^n \lambda(D_k)$ denote their average landscape.
We can summarize this construction as a mapping
\begin{equation}
  \label{eq:al}
  (D_1,\ldots,D_n) \mapsto \bar{\lambda} = \bar{\lambda}(D_1,\ldots,D_n)
\end{equation}
We will show that in many cases, this map is invertible.

\subsection{Noninvertibility and connected persistence diagrams}
\label{sec:noninvertibility}

We start with a simple example where the map in \eqref{eq:al} is not one-to-one and hence not invertible.

Consider $D_1 = \{(0,2)\}$, $D_2 = \{(4,6)\}$, $D'_1 = \{(0,2),(4,6)\}$, and $D'_2 = \emptyset$.
Then $\lambda(D_1) + \lambda(D_2) = \lambda(D'_1) = \lambda(D'_1) + \lambda(D'_2)$.
So the average landscape of $\{D_1,D_2\}$ equals the average landscape of $\{D'_1,D'_2\}$.

The map \eqref{eq:al} fails to be invertible because the union of the intervals in the bar code (Section~\ref{sec:persistence}) corresponding to the persistence diagram $D'_1$ is disconnected.
However, in many applications we claim that this behavior is atypical. To make this claim precise we need the following definition.

\begin{definition} \label{def:graph}
  Let $B$ be a bar code consisting of intervals $\{I_j\}_{j \in J}$. Define the \emph{graph} of $B$ to be the graph whose vertices are the intervals $I_j$ and whose edges $\{I_j,I_k\}$ consists of pairs of intervals with nonempty intersection, $I_j \cap I_k \neq \emptyset$.
\end{definition}

For many geometric processes~\cite[Figure 2.2]{Adler:2010} and in applications~\cite[Figure 5]{Gameiro:2015b}, as the number of intervals in the bar code increases, the corresponding graphs seem to have a \emph{giant component}~\cite[Chapter 6]{Bollobas:book}.
Note that any gaps in the union of intervals in the bar code only occur where the corresponding Betti number is zero. So there will be no gaps in a range of parameter values where all of the corresponding Betti numbers are nonzero.

\subsection{Bipartite graph of a persistence diagram}
\label{sec:bipartite-graph}

Let $D = \{(a_j,b_j)\}_{j \in J}$ be a persistence diagram.

\begin{definition} \label{def:generic-pd}
  Say that the persistence diagram $D$ is \emph{generic} if for each $j \neq k \in J$, the four numbers $a_j,b_j,a_k,b_k$ are distinct.
\end{definition}

\begin{definition} \label{def:bipartite}
  Let $D$ be a generic persistence diagram. Let $B(D)$ be the \emph{bipartite graph} of $D$ consisting of the disjoint vertex sets $U = \{a_j\}_{j \in J}$ and $V = \{b_j\}_{j \in J}$ and edges consisting of $(a_j,b_j)$ for each $j \in J$ and $(a_k,b_j)$ for each pair $j,k \in J$ satisfying $a_j < a_k < b_j < b_k$.
\end{definition}

\begin{proposition} \label{prop:reconstruct}
  We can reconstruct a generic persistence diagram $D$ from its bipartite graph.
\end{proposition}

\begin{proof}
  Let $D$ be a generic persistence diagram. Let $B(D)$ be its bipartite graph.
  Let $U$ and $V$ be the disjoint vertex sets of $B(D)$.
  By definition, $U$ consists of the set of first coordinates of the points in $D$, and $V$ consists of the set of second coordinates of the points in $D$.
  By assumption, these coordinates are unique.
  Let $a \in U$.
  By the definition of $B(D)$, there exists $b \in V$ such that $\{a,b\}$ is an edge in $B(D)$ and $(a,b) \in D$.
  Also by definition, for all $c \in V$ such at $\{a,c\}$ is an edge in $B(D)$, $c \leq b$.
  Thus, for all $a \in U$, let $b = b(a)$ be the maximum element of $V$ such that $\{a,b\}$ is an edge in $B(D)$.
  The resulting pairs $(a,b)$ are exactly $D$.
\end{proof}

\begin{definition}
  Say that a persistence diagram is \emph{connected} if the graph~(Definition~\ref{def:graph}) of its barcode is connected.
\end{definition}

\begin{lemma}
   A generic persistence diagram is connected if and only if its bipartite graph is connected.
\end{lemma}

\begin{proof}
  Let $D$ be a generic persistence diagram. If we set $a \sim b$ in $(a,b) \in D$, the $B(D)/\sim$ is isomorphic to the graph of the bar code corresponding to $D$.
By definition, $B(D)$ is connected, if and only if $B(D)/\sim$ is connected.
\end{proof}


\subsection{Critical points of persistence landscapes}
\label{sec:critical-points}

We observe that it is easy to list the critical points of a persistence landscape from its corresponding persistence diagram.

\begin{lemma}
Let $D = \{(a_j,b_j)\}$ be a persistence diagram.
Consider the intervals $[a_j,b_j)$ in the corresponding bar code.
The critical points in the corresponding persistence landscape consist of
\begin{enumerate}
\item \label{it:1} the left end points $a_j$ of the intervals;
\item \label{it:2} the right end points $b_j$ of the intervals;
\item \label{it:3} the midpoints $\frac{a_j+b_j}{2}$ of the intervals; and
\item \label{it:4} the midpoints $\frac{a_k+b_j}{2}$ of intersections of pairs of intervals where $a_j < a_k < b_j < b_k$.
\end{enumerate}
Let $C(D)$ denote this set.
\end{lemma}

\begin{proof}
  Recall that the critical points of the persistence landscape of $D = \{(a_j,b_j)\}_{j \in J}$ consist of the critical points of the functions $f_{(a_j,b_j)}$ and the points $t$ for which there exist $j$ and $k$ such that 
    $f_{(a_j,b_j)}(t) = f_{(a_k,b_k)}(t)$,
    $f'_{(a_j,b_j)}(t) = -1$ and $f'_{(a_k,b_k)}(t)=1$.
    The former are exactly the points in \eqref{it:1}, \eqref{it:2}, and \eqref{it:3}.
    The latter are exactly the points in \eqref{it:4}.
\end{proof}

In the set $C(D)$ we have the following three-term arithmetic progressions,
\begin{equation*}
  a_j,\tfrac{a_j+b_j}{2},b_j \quad \text{and} \quad a_k,\tfrac{a_k+b_j}{2},b_j,
\end{equation*}
which we call \emph{interval triples} and \emph{intersection triples}, respectively.
Note that we have one interval triple for each point in the persistence diagram and one intersection triple for each pair of points in the persistence diagram that satisfies $a_j < a_k < b_j < b_k$.

\subsection{Arithmetically independent sets of persistence diagrams}

In this section we introduce assumptions for a set $\{D_1,\ldots,D_n\}$ of persistence diagrams.

\begin{definition} \label{def:arithmetically-independent}
  Let $\{D_1,\ldots,D_n\}$ be a set of persistence diagrams.
  We call this set \emph{arithmetically independent} if it satisfies the following assumptions.
  \begin{enumerate}
  \item Each $D_i$ is generic. 
  \item The sets $C(D_i)$ are pairwise disjoint.
  \item Let $C$ be the set of all critical points in $\bar{\lambda}(D_1,\ldots,D_n)$. All of the three-term arithmetic progressions in $C$ are either interval triples or intersection triples of some $D_i$.
  \end{enumerate}
\end{definition}

\begin{example}
  The set $\{D\}$, where $D = \{(0,1),(1,2)\}$, is not arithmetically independent since $1$ appears twice as an endpoint of an interval in $D$.
  The set $\{D_1,D_2\}$, where $D_1 = \{(0,2)\}$ and $D_2 = \{(1,5)\}$, is not arithmetically independent since $1$ is a midpoint of an interval in $D_1$ and an endpoint of an interval in $D_2$.
  The set $\{D_1,D_2\}$, where $D_1 = \{(0,1)\}$ and $D_2 = \{(2,4)\}$ is not arithmetically independent because of the three-term arithmetic progression $(0,1,2)$.
  The set $\{D_1,D_2\}$, where $D_1 = \{(0,8)\}$ and $D_2 = \{(11,13)\}$ is not arithmetically independent because of the three-term arithmetic progression $(4,8,12)$.
  However, if we add $0$, $0.1$, $0.01$, and $0.001$ to the four respective numbers in each of these examples, then they become arithmetically independent.
\end{example}

\subsection{Reconstruction of persistence diagrams from an average landscape} \label{sec:reconstruct}

We are now in a position to state and prove our reconstruction result.

\begin{theorem} \label{thm:reconstruct}
  Let $\bar{\lambda}$ be the average landscape of the persistence diagrams $D_1,\ldots,D_n$.
  If $D_1,\ldots,D_n$ are connected and arithmetically independent then one can
  reconstruct $\{D_1,\ldots,D_n\}$ from $\bar{\lambda}$.
\end{theorem}

\begin{proof}
  Let $C$ be the set of all critical points in the average landscape $\bar{\lambda}(D_1,\ldots,D_n)$.
  Let $U \subset C$ be the subset of critical points that are the first term in a three-term arithmetic progression in $C$.
  Let $V \subset C$ be the subset of critical points that are the third term in a three-term arithmetic progression in $C$.

  By assumption $U$ and $V$ are disjoint.
  Let $B$ be the bipartite graph whose set of vertices is the disjoint union of $U$ and $V$ and whose edges consist of $\{a,b\}$ where $a$ and $b$ are the first and third term of a three-term arithmetic progression in $C$.

  By the assumption of arithmetic independence, vertices in $B$ are only connected by an edge if they are critical points of the same persistence diagram.
  By the assumption of connectedness, all of the critical points of a persistence landscape of one of the persistence diagrams are connected in $B$.
  Thus, the connected components of $B$ are exactly the bipartite graphs $B(D_1),\ldots,B(D_n)$.
  
  Using Proposition~\ref{prop:reconstruct}, we can reconstruct each persistence diagram from the corresponding bipartite graph.
\end{proof}

\subsection{Persistence landscapes are characteristic for empirical measures} \label{sec:finitely-characteristic}

We can restate Theorem~\ref{thm:reconstruct} using the language of characteristic kernels (Section~\ref{sec:kernel}).

\begin{theorem} \label{thm:characteristic}
  The persistence landscape kernel is characteristic for empirical measures on connected and arithmetically independent persistence diagrams.
\end{theorem}

\subsection{Genericity of arithmetically independent persistence diagrams}
\label{sec:generic}

We end this section by showing that connected and arithmetically independent persistence diagrams are generic in a particular sense.

\begin{lemma} \label{lem:connected}
  Let $D = \{(a_j,b_j)\}_{j=1}^n$ be a persistence diagram. Let $\eps>0$.
  Then there exists a connected persistence diagram $D'$ with $d_B(D,D') < \eps$.
\end{lemma}

\begin{proof}
  Let $a = \min\{a_j\}$ and $b = \max\{b_j\}$.
  Choose $N$ such that $\frac{b-a}{N} < \frac{\eps}{2}$.
  Let $D'' = \{(a+ (k-1)\frac{b-a}{N}, a + (k+1)\frac{b-a}{N})\}_{k=0}^N$.
  Then $D''$ is connected and $d_B(D'',\emptyset) < \eps$.
  Thus $D \amalg D''$ is connected and $d_B(D,D \amalg D'')< \eps$.
\end{proof}

\begin{lemma} \label{lem:generic}
  Let $D = \{(a_j,b_j)\}_{j=1}^n$. Let $\eps>0$.
  Then there is a generic persistence diagram $D' = \{(a'_j,b'_j)\}_{j=1}^n$ with $d_B(D,D') < \eps$.
Furthermore, if $D$ is connected then so is $D'$.
\end{lemma}

\begin{proof}
  The proof is by induction on $n$.
  If $n=0$ then the statement is trivial.
  Assume that $\{(a'_j,b'_j)\}_{j=1}^{n-1}$ is a generic persistence diagram and $d_B(\{(a_j,b_j)\}_{j=1}^{n-1},\{(a'_j,b'_j)\}_{j=1}^{n-1}) < \eps$.
  Since there are only finitely many numbers to avoid, we can choose
  $a'_n \in [a_n-\frac{\eps}{2},a_n]$ and $b'_n \in [b_n,b_n+\frac{\eps}{2}]$ such that $D' := \{(a'_j,b'_j)\}_{j=1}^n$ is a generic persistence diagram.
  Note that $d_B(D,D') < \eps$.
  Since $a'_n \leq a_n < b_n \leq b'_n$, if $D$ is connected then so is $D'$.
\end{proof}

\begin{proposition} \label{prop:generic}
  Let $D$ be a generic persistence diagram.
  Then there is an $\eps>0$ such that for all $D'$
  with $\abs{D'}=\abs{D}$ and $d_B(D,D')< \eps$, $D'$ is generic and $B(D') \isom B(D)$.
\end{proposition}

\begin{proof}
  Let $E(D)$ be the set of all coordinates of points in $D$.
  Let $\delta = \min\{\abs{x-y} \st x \neq y \in E(D)\}$.
  Let $\eps < \frac{\delta}{4}$.
  Let $D'$ be a persistence diagram with $\abs{D'}=\abs{D}$ and $d_B(D,D')< \eps$.
  Then for all $(a,b) \in D$ there is a $(a',b') \in D'$ with $\norm{(a,b)-(a',b')}_{\infty} < \eps$.
  So $\abs{a'-a} < \eps$ and $\abs{b'-b} < \eps$.
  By the triangle inequality, the coordinates of points in $D'$ are distinct.

  By the construction of $D'$, there is a canonical bijection of the intervals in the barcodes of $D$ and $D'$.
  Note that by the definition of $\delta$, this implies that the nonempty intersections of pairs of intervals in the bar code of $D$ have length at least $\delta$.
  Since $\eps < \frac{\delta}{4}$, a pair of intervals in the bar code of $D'$ intersect if and only if the corresponding pair of intervals in $D$ intersect.
\end{proof}

\begin{corollary} \label{cor:connected2}
  Let $D$ be a generic and connected persistence diagram.
  Then there is an $\eps>0$ such that for all persistence diagrams $D'$ with
  $\abs{D'} = \abs{D}$ and $d_B(D,D')<\eps$,
  $D'$ is generic and connected.
\end{corollary}

Now consider a sequence of persistence diagrams $D_1,\ldots,D_n$. Recall that we consider this to be a point in the product space of $n$ persistence diagrams (Section~\ref{sec:persistence}) with associated product metric \eqref{eq:product-metric} and product topology.

\begin{theorem} \label{thm:generic}
  Connected and arithmetically independent persistence diagrams are generic in the following sense. 
  \begin{enumerate}
  \item \label{it:generic1} They are dense. That is, given persistence diagrams $D_1,\ldots,D_n$ and an $\eps >0$ there exist connected and arithmetically independent persistence diagrams $D'_1,\ldots,D'_n$ with $d_B(D_i,D'_i) < \eps$ for all $i$.
  \item \label{it:generic2} If we restrict to persistence diagrams with the same cardinality then they are open.
    That is, given connected and arithmetically independent persistence diagrams $D_1,\ldots,D_n$,  there is some $\eps>0$ such that
    any persistence diagrams $D'_1,\ldots,D'_n$ with $\abs{D'_i} = \abs{D_i}$ and $d_B(D_i,D'_i) < \eps$ for all $i$, are connected and arithmetically independent. 
  \end{enumerate}
\end{theorem}

\begin{proof}
  (\ref{it:generic1}) The proof is by induction on $n$.
  If $n=0$ then the statement is trivially true.
  Assume that we have connected and arithmetically independent persistence diagrams $D'_1,\ldots,D'_{n-1}$ with $d_B(D_j,D'_j) < \eps$ for $1 \leq j \leq n-1$.
  By Lemmas \ref{lem:connected} and \ref{lem:generic} there exists a generic and connected persistence diagram $D'_n = \{(a_k,b_k)\}_{k=1}^m$ with $d_B(D_n,D'_n) < \frac{\eps}{2}$.
  We finish the proof by induction on $m$.
  If $m=0$ then we are done.
  Assume that $D'_1,\ldots,D'_{n-1},\{(a'_k,b'_k)\}_{k=1}^{m-1}$ is arithmetically independent. 
  By Corollary~\ref{cor:connected2}, there exists an $\eps'>0$ such that
   for all persistence diagrams $D''$ with
  $\abs{D''} = m$ and $d_B(D'',D'_n)<\eps'$,
  $D''$ is generic and connected.
  Let $\delta = \min(\frac{\eps}{4},\frac{\eps'}{2})$.
  Since there are only finitely many numbers to avoid, we can choose
  $a'_m \in [a_m-\delta]$ and $b'_m \in [b_m,b_m+\delta]$ such that $D'_1,\ldots,D'_{n-1},D_n'':=\{(a_k,b_k)\}_{k=1}^{m-1} \cup \{(a'_m,b'_m)\}$ is connected and arithmetically independent.
  Note that $d_B(D_n,D''_n) < \eps$.
   
  (\ref{it:generic2}) Let $D_1,\ldots,D_n$ be connected persistence diagrams that are arithmetically independent. Denote this sequence of persistence diagrams by $\mathcal{D}$. Using Corollary~\ref{cor:connected2} we can choose an $\eps'>0$ such that 
  for any persistence diagrams $D'_1,\ldots,D'_n$ with $\abs{D'_i} = \abs{D_i}$ and $d_B(D_i,D'_i) < \eps'$ for all $i$, each $D'_i$ is connected.

  Let $C(\mathcal{D})$ be the set of all critical points of the average landscape of $\mathcal{D}$.
  There are only finitely many points $a \in \R \setminus C(\mathcal{D})$ such that $a$ is part of a three term arithmetic progression in $C(\mathcal{D}) \cup \{a\}$.
  Let $C'(\mathcal{D})$ be the set of all such numbers. 

  Let $\delta = \min\{ \abs{x-y} \st x \neq y \in C(\mathcal{D}) \amalg C'(\mathcal{D})\}$.
  Let $\eps'' = \frac{\delta}{4}$.
  Consider persistence diagrams $D'_1,\ldots,D'_n$ with $\abs{D'_i} = \abs{D_i}$ and $d_B(D_i,D'_i) < \eps''$ for all $i$.
  Let $\mathcal{D'}$ denote this sequence of persistence diagrams.

  The assumptions imply that for each point $(a,b)$ in one of the persistence diagrams in $\mathcal{D}$ there is a corresponding point $(a',b')$ in the corresponding persistence diagram in $\mathcal{D'}$, and $\norm{(a,b)-(a',b')}_{\infty} < \eps''$. That is, $\abs{a-a'}<\eps''$ and $\abs{b-b'}<\eps''$.
  Thus we have the induced bijection between $C(\mathcal{D})$ and $C(\mathcal{D'})$ with corresponding points $x$ and $x'$ satisfying $\abs{x-x'}<\eps''$.
  Notice that since $D_i$ is generic, so is $D'_i$.
  Also, since the sets $S(D_i)$ are disjoint, so are the sets $S(D'_i)$.
  Furthermore, the assumptions imply that we have an induced correspondence between $C'(\mathcal{D})$ and $C'(\mathcal{D'})$ with corresponding points $y$ and $y'$ satisfying $\abs{y-y'}<2\eps''$.
  By the triangle inequality for $x' \in C(\mathcal{D'})$, $y' \in C'(\mathcal{D'})$, $\abs{x'-y'} > \delta - 3\eps'' > \eps''$.
  It follows that $\mathcal{D}$ is arithmetically independent.
  Let $\eps = \min(\eps',\eps'')$.
\end{proof}

\section{Metric comparison of persistence landscapes and persistence diagrams} \label{sec:metric}

In this section we show that the $L^{\infty}$ landscape distance can be much smaller than the corresponding bottleneck distance.

Given a persistence diagram $D$, let $\lambda(D)$ denote the corresponding persistence landscape.
In \cite[Theorem 12]{bubenik:landscapes} it was shown that $\norm{\lambda(D)-\lambda(D')}_{\infty} \leq d_B(D,D')$.

Here we will show the following.

\begin{proposition}
  Let $K > 0$. Then there is a pair of persistence diagrams such that
$\norm{\lambda(D)-\lambda(D')}_{\infty} \leq K d_B(D,D')$.
\end{proposition}

\begin{proof}
  Consider
  \begin{gather*}
    D_1 = \{\pm(-3n-1+2i,3n-1+2i))\}_{i=1}^n, \text{ and }\\
    D_2 = \{\pm(-3n+2i,3n+2i)\}_{i=1}^{n-1} \cap \{(-3n,3n),(-n,n)\}
  \end{gather*}
  See Figure~\ref{fig:counter-example} where $n=4$.
  Then $\norm{\lambda(D_1)-\lambda(D_2)}_{\infty} = 1$, but $d_B(D_1,D_2) = 2n+1$.
\end{proof}

\begin{figure}
  \centering
  \begin{tikzpicture}[scale=0.3]
    \def\n{4}
    \draw[->] (-5*\n-1,0) -- (5*\n+1,0); 
    \foreach \x in {-20,...,20}
      \draw (\x,0.1) -- (\x,-0.1);
    \foreach \x in {-20,-15,...,20}
      \draw (\x,0) node[anchor=north] {\x};
    \foreach \i in {1,2,...,\n}
    {
      \draw (-3*\n - 1 + 2*\i, 0) -- (-1+2*\i,3*\n) -- (3*\n -1 +2*\i,0);
      \fill (-1+2*\i,3*\n) circle (2mm);
      \draw (3*\n + 1 - 2*\i, 0) -- (1-2*\i,3*\n) -- (-3*\n +1 -2*\i,0);
      \fill (1-2*\i,3*\n) circle (2mm);
    }
    \foreach \i in {0,1,...,3}
    {
      \draw [dashed] (-3*\n + 2*\i, 0) -- (2*\i,3*\n) -- (3*\n +2*\i,0);
      \draw (2*\i,3*\n) circle (2mm);
      \draw [dashed] (3*\n - 2*\i, 0) -- (-2*\i,3*\n) -- (-3*\n -2*\i,0);
      \draw (-2*\i,3*\n) circle (2mm);
    }
    \draw [dashed] (-\n, 0) -- (0,\n) -- (\n,0);
    \draw (0,\n) circle (2mm);
  \end{tikzpicture}
  \caption{Two persistence diagrams, $D_1$ and $D_2$ (with filled circles and open circles, respectively) whose persistence landscape distance is much smaller than their bottleneck distance. Each point $(b,d)$ in the persistence diagram is plotted with coordinates $(m,h)$, where $m=\frac{b+d}{2}$ and $h=\frac{d-b}{2}$. The corresponding persistence landscapes, $\lambda(D_1)$ and $\lambda(D_2)$ are given by solid and dashed lines respectively. Observe that $\norm{\lambda(D_1)-\lambda(D_2)}_{\infty} = 1$ but $d_B(D_1,D_2) = 9$.}
  \label{fig:counter-example}
\end{figure}
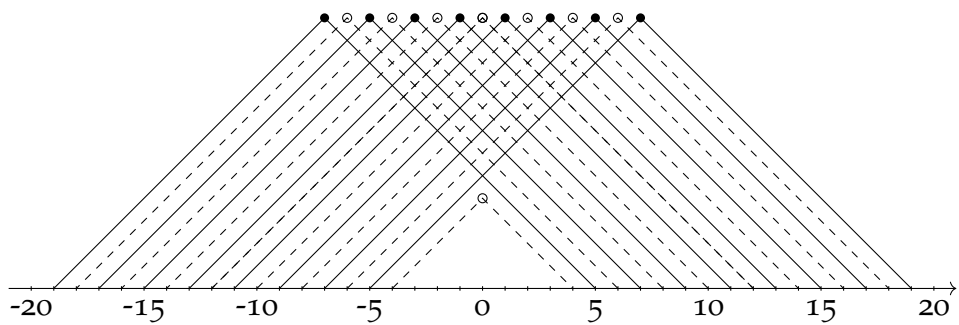
 
\subsection*{Acknowledgments}

The author would like to acknowledge the support of the Army Research Office [Award W911NF1810307], National Science Foundation [DMS - 1764406] and the Simons Foundation [Grant number 594594].
He would also like to thank
Pawel Dlotko, Michael Kerber, and Oliver Vipond for helpful conversations, 
Leo Betthauser, Nikola Milicevic, and Alex Wagner for proofreading an earlier draft,
and the Mathematisches Forschungsinstitut Oberwolfach (MFO) where some of this work was started.

 

\end{document}